\documentclass[11pt]{article}
\usepackage{amssymb}
\usepackage[a4paper,margin=1in]{geometry}
\usepackage{amsmath,amssymb,amsthm,mathtools}
\usepackage[hidelinks]{hyperref}
\usepackage{microtype}

\newtheorem{theorem}{Theorem}
\newtheorem{lemma}[theorem]{Lemma}

\theoremstyle{definition}
\newtheorem{definition}[theorem]{Definition}
\newtheorem{example}[theorem]{Example}
\theoremstyle{remark}

\newcommand{\F}{\mathbb F}
\newcommand{\R}{\F_q[x]}
\newcommand{\V}{\mathcal V}
\newcommand{\C}{\mathcal C}
\newcommand{\PP}{\mathbb P}

\newcommand{\1}{\mathbf 1}

\title{An Asymptotic Bound for Non-covering Congruence Systems over
  \texorpdfstring{$\mathbb F_q[x]$}{Fq[x]}}
\author{Rongyin Wang}
\date{\today}

\begin{document}

\maketitle

\begin{abstract}
Fix a prime power $q$.  Let $D_q(n)$ be the largest possible least degree
of a polynomial omitted by a non-covering family of $n$ congruence classes
in $\mathbb F_q[x]$.  We prove
\[
  D_q(n)=\frac{n}{q-1}+O_q(1).
\]
The upper bound combines a minimal-counterexample reduction to irreducible
moduli with a truncated inclusion--exclusion (Brun sieve) argument.  A
nested-modulus construction gives the matching lower bound. This is a follow-up to the author's earlier work.
\end{abstract}

\section*{Notations and Conventions}

Throughout, $q$ is a fixed prime power and
\[
  s:=q-1.
\]
For an integer $m\geq 0$, write
\[
  \V_m:=\{f\in\R:\deg f<m\}.
\]
We use the convention $\deg 0=-\infty$, so that $0\in\V_m$ for every
$m\geq 0$.  Thus $\lvert\V_m\rvert=q^m$.

\section{Introduction}

Let $\F_q$ be a finite field and take its (univariate) polynomial ring $\F_q[x]$. A congruence class is a coset of an ideal $(f)\subseteq\R$. Following the usual notation in $\mathbb{Z}$, we write

\begin{definition}
    A congruence class is a set
\[
  r \pmod{f}
  :=
  \{g\in\R:g\equiv r\pmod f\},
\]
where $f$ is nonconstant.  A finite family of congruence classes is
\emph{non-covering} if its union is not all of $\R$.
\end{definition}

\begin{definition}
For $n\geq 1$, define
\[
  D_q(n):=
  \max_{\substack{\C\text{ is a non-covering family}\\
                   \lvert\C\rvert=n}}
  \min\{\deg g:g\notin\bigcup\C\}.
\]
\end{definition}

In the author's earlier work \cite{MR4954327}, we studied the analogue of Erd\H{o}s' covering system problem in the setting of polynomial rings over finite fields. We obtained a loose bound $D_q(n)< n$.

In the end of \cite{MR4954327}, it is conjectured that a sharper bound could be reached where $D_q(n)=\lfloor \frac{n}{s}\rfloor$, but the conjecture is false. For example, we have:

\begin{example}\label{ctex}
    Choose $\F_q=\F_7$ and a congruence system $B$ of 11 classes:
    \[
\begin{cases}
 0,1,4,6 \pmod{x},
\\
 6 \pmod{x-2},
\\
5 \pmod{x-3}
\\
4 \pmod{x-4}
\\
2,3,4,6 \pmod{x-6}
\end{cases}
\]
Then $B$ covers all polynomials of degree $<2=\lfloor\frac{11}{6}\rfloor+1$ in $\F_7[x]$ but not $x^2+2x+2$.
\end{example}

The main result of this paper shows the leading term $D_q(n)\sim \frac{n}{q-1}$ is correct, although we need an error term depending only on $q$. The main tool in our paper is to build a truncated sieve, separating our congruence classes into "low-modulus" ones which are larger but mutually independent and "high-modulus" ones which are smaller, so that we can control the probability of a uniformly chosen polynomial being sifted.

The rest of the paper is organised as follows:

In section \ref{thmsec} we state the main theorem of the paper,

In section \ref{lm1sec} we generalise Lemma 5 in \cite{MR4954327} to fit the asymptotic bound in this paper,

In section \ref{lem2sec} we build a Brun sieve on the polynomial ring with Chinese Remainder Theorem.

In section \ref{pfsec} we prove the main theorem.

\bigskip

\textbf{Acknowledgement.} I am very grateful to Dr Huixi Li of Nankai University for introducing me to Erdős's original conjecture and its analogue for polynomial rings over finite fields. I would also like to thank Dr Thomas Bloom of the University of Manchester for suggesting that I explore probabilistic methods to obtain the asymptotic bound.

\section{Statement of the Main Theorem}\label{thmsec}

We use the following previously established result.

\begin{theorem}[\cite{MR4954327}, theorem 2]\label{thm:known}
If a family of $n$ congruence classes covers $\V_n$, then it covers
$\R$.
\end{theorem}

Our main result is the following.

\begin{theorem}\label{thm:main}
For every fixed prime power $q$, there is a constant $C_q$ such that
\[
  \left\lfloor\frac{n-1}{q-1}\right\rfloor
  \leq D_q(n)
  \leq
  \left\lceil\frac{n}{q-1}\right\rceil+C_q-1
\]
for every $n\geq 1$.  In particular,
\[
  D_q(n)=\frac{n}{q-1}+O_q(1)
  \qquad\text{and hence}\qquad
  D_q(n)\sim\frac{n}{q-1}.
\]
\end{theorem}

\section{A generalized minimal-counterexample reduction}\label{lm1sec}

Fix an integer $C\geq 1$, and define
\[
  L_C(n):=\left\lceil\frac{n}{s}\right\rceil+C.
\]

\begin{lemma}[Irreducible-modulus reduction]\label{lem:reduction}
Suppose that $n$ is minimal with the following property: there is a family
of $n$ congruence classes which covers $\V_{L_C(n)}$ but does not cover
$\R$.  Then there is such a family with the following additional
properties.
\begin{enumerate}
  \item Every modulus is irreducible.
  \item If $p$ is irreducible of degree $d$, then at most $s\cdot d$ of the
  classes have modulus $p$.
\end{enumerate}
\end{lemma}

\begin{proof}
Let $\C$ be a family with $n$ minimal, and choose
$\alpha\in\R\setminus\bigcup \C$.

\medskip
\noindent
First, suppose a class has modulus
\[
  f=\prod_{j=1}^u p_j^{e_j}.
\]
Because $\alpha\not\equiv r\pmod f$, there is some $j$ such that
$\alpha\not\equiv r\pmod{p_j^{e_j}}$.  Replacing $r\pmod f$ by
$r\pmod{p_j^{e_j}}$ enlarges the class, preserves coverage of
$\V_{L_C(n)}$, and still omits $\alpha$.  We may therefore assume that
every modulus is a prime power. Additionally, if $\alpha\not\equiv r \pmod {p_j}$, we can replace $r\pmod {p_j^{e_j}}$ with $r \pmod {p_j}$. We claim this won't create duplicated classes, otherwise we can remove it and obtain a smaller family, contradiction.

\medskip
\noindent
Let $k_p$ be the number of classes whose modulus is exactly $p$ and $\alpha_p$ be such that $\alpha_p\in \V_{d}, \alpha_p\equiv \alpha\pmod p$.  Their
residues are distinct and different from $\alpha_p$.

Assume first that $d<L_C(n)$.  

There are three types of congruence classes:  
\[
\begin{cases}
r_i(x) \pmod{p(x)} \text{ such that }
\alpha(x) \not\equiv r_i(x) \pmod{p(x)}& (a), 
\\
\kappa(x) p(x) +\alpha_p(x) \pmod{p(x)^{\lambda}}
\text{ for some }
\lambda>1& (b), 
\\
r'_i(x) \pmod{u(x)^e} 
\text{ such that }
\gcd(p(x),u(x)^e)=1 & (c).
\end{cases}
\]

The map
\[
   \tau:\V_{L_C(n)-d} \rightarrow \V_{L_C(n)}, h\longmapsto \alpha_p+ph
\]
is a bijection from $\V_{L_C(n)-d}$ onto the elements of
$\V_{L_C(n)}$ which are congruent to $\alpha_p$ modulo $p$.  None of
the $k_p$ type-$(a)$ classes meets this slice.

\begin{itemize}
  \item A type-(b) class
  \[
    \alpha_p+p\kappa\pmod{p^e},\qquad e>1,
  \]
  induces a class $C:=\kappa\pmod{p^{e-1}}$;
  \item For a type-(c) class $C_0:=r\pmod{u^e}$, where $u\neq p$, taking $\tau^{-1}(C_0\cap (\alpha_p \pmod{p}))$ gives us a class
  \[
    p^{-1}(r-\alpha_p)\pmod{u^e}.
  \]
\end{itemize}
Thus the $n-k_p$ classes in total form a non-covering family which
covers $\V_{L_C(n)-d}$.  The induced family is non-covering because it
omits $(\alpha-\alpha_p)/p$.

If $k_p\geq sd$, then
\[
  n-k_p\leq n-sd
  \leq
  s\left(\left\lceil\frac ns\right\rceil-d\right),
\]
and hence
\[
  \left\lceil\frac{n-k_p}{s}\right\rceil
  \leq
  \left\lceil\frac ns\right\rceil-d.
\]
Consequently,
\[
  L_C(n-k_p)\leq L_C(n)-d.
\]
After deleting any duplicate induced classes, if necessary, we obtain a
counterexample with fewer than $n$ classes, contradicting minimality.
Therefore
\[
  k_p<sd.
  \tag{3.1}\label{eq:simple-bound}
\]

If $d\geq L_C(n)$, then
\[
  sd\geq sL_C(n)
  \geq n+sC>n,
\]
so $k_p\leq n<sd$.  Hence \eqref{eq:simple-bound} holds for every $p$.

\medskip
\noindent

Back to assertion 1. Fix irreducible $p$ and let $b_p$ be the number of remaining classes whose
moduli are $p^e$ with $e>1$.  Replace all these $b_p$ classes, when
$b_p>0$, by the single class
\[
  \alpha_p\pmod p.
\]
This preserves coverage of $\V_{L_C(n)}$.

There is a residue $\beta_p\in\R/(p)$ which is different from
$\alpha_p$ and from all $k_p$ simple selected residues.  Indeed,
\eqref{eq:simple-bound} gives
\[
  k_p+1\leq sd,
\]
while
\[
  q^d=(1+s)^d\geq 1+sd.
\]
Thus at most $q^d-1$ residues have been excluded.

Choose such a $\beta_p$ for every irreducible $p$ appearing in the
family.  By the Chinese Remainder Theorem, there is a polynomial $\beta$
such that
\[
  \beta\equiv\beta_p\pmod p
\]
for every such $p$.  The polynomial $\beta$ avoids every class in the
modified family, so that family is still non-covering.

The number of classes in the modified family is
\[
  n'
  =
  \sum_p\bigl(k_p+\1_{\{b_p>0\}}\bigr)
  \leq
  \sum_p(k_p+b_p)
  =n.
\]
If $n'<n$, then the modified family covers $\V_{L_C(n)}$, and therefore
also $\V_{L_C(n')}$, while remaining non-covering.  This contradicts the
minimality of $n$.  Thus $n'=n$, which forces $b_p\leq 1$ for every
$p$.

After the replacements, all moduli are irreducible.  The number $K_p$ of
selected residues modulo $p$ satisfies
\[
  K_p=k_p+\1_{\{b_p=1\}}\leq sd.
\]
The residue $\beta_p$ is not selected, so the selected residue set is
proper.  This proves both assertions.
\end{proof}

\section{A truncated CRT sieve}\label{lem2sec}

\begin{lemma}[Truncated CRT sieve]\label{lem:sieve}
For every fixed $q$, there are constants $H_q>0$ and $N_q$ with the
following property.  Let $n\geq N_q$, and let $\mathcal P$ be a finite
set of distinct monic irreducible polynomials.  For each
$p\in\mathcal P$, let
\[
  B_p\subsetneq\R/(p),\qquad
  k_p:=|B_p|,\qquad
  d_p:=\deg p,
\]
and suppose that
\[
  k_p\leq sd_p,
  \qquad
  \sum_{p\in\mathcal P}k_p=n.
  \tag{4.1}\label{eq:sieve-assumptions}
\]
If
\[
  m\geq H_q\left\lceil\log_q(n+2)\right\rceil^2,
  \tag{4.2}\label{eq:sieve-range}
\]
then there is an $f\in\V_m$ such that
\[
  f\bmod p\notin B_p
  \qquad\text{for every }p\in\mathcal P.
\]
\end{lemma}

\begin{proof}
We use a probabilistic method. Choose $f$ uniformly from $\V_m$.  For $p\in\mathcal P$, define the events
\[
  E_p:=\{f\bmod p\in B_p\}.
\]
If $d_p\leq m$, then the reduction map
$\V_m\to\R/(p)$ is uniform, so
\[
  \lambda_p:=\PP(E_p)=\frac{k_p}{q^{d_p}}.
  \tag{4.3}\label{eq:lambda}
\]

More generally, if $p_1,\dots,p_j$ are distinct and
\[
  d_{p_1}+\cdots+d_{p_j}\leq m,
\]
then reduction modulo $p_1\cdots p_j$ is uniform on $\V_m$.  The Chinese
Remainder Theorem therefore gives the exact identity
\[
  \PP(E_{p_1}\cap\cdots\cap E_{p_j})
  =
  \lambda_{p_1}\cdots\lambda_{p_j}.
  \tag{4.4}\label{eq:limited-independence}
\]

Put
\[
  t:=\left\lceil\log_q(n+2)\right\rceil.
\]
We shall choose constants $K_q,R_q>0$, depending only on $q$, and set
\[
  Z:=\lceil K_qt\rceil.
\]
Call $p$ \emph{low} if $d_p\leq Z$, and set
\[
  P:=\prod_{d_p\leq Z}(1-\lambda_p),
  \qquad
  \mu:=\sum_{d_p\leq Z}\lambda_p.
  \tag{4.5}\label{eq:P-mu}
\]

\medskip
\noindent
\textbf{A lower bound for $P$ and an upper bound for $\mu$.}
Let $I_q(d)$ be the number of monic irreducible polynomials of degree
$d$ over $\F_q$.  Since each such polynomial has $d$ distinct roots
in $\F_{q^d}$, and distinct irreducibles have disjoint root sets,
\[
  I_q(d)\leq\frac{q^d}{d}.
  \tag{4.6}\label{eq:irreducibles}
\]
It follows from \eqref{eq:sieve-assumptions} that
\[
  \sum_{\substack{p\in\mathcal P\\d_p=d}}\lambda_p
  \leq
  \frac{q^d}{d}\frac{sd}{q^d}
  =s.
  \tag{4.7}\label{eq:density-per-degree}
\]

Choose $d_0=d_0(q)$ so that
\[
  \frac{sd}{q^d}\leq\frac12
  \qquad(d\geq d_0).
\]
For $d_p\geq d_0$, we have
\[
  -\log(1-\lambda_p)\leq 2\lambda_p,
  \tag{4.8}\label{eq:log-bound} 
\]
by considering $-\log(1-x)-2x\leq0,0\leq x\leq\frac{1}{2}$.

For $d_p<d_0$, the assumption $B_p\subsetneq\R/(p)$ gives
\[
  1-\lambda_p\geq q^{-d_p}.
\]
There are only finitely many such $p$, so their total contribution to
$-\log P$ is bounded by a constant depending only on $q$.

Using \eqref{eq:density-per-degree} and \eqref{eq:log-bound} for
$d_0\leq d\leq t$, we obtain a contribution at most $2st+O_q(1)$.
For $t<d_p\leq Z$, we have
\[
  \sum_{t<d_p\leq Z}\lambda_p
  \leq
  q^{-t}\sum_{p\in\mathcal P}k_p
  \leq nq^{-t}\leq 1.
\]
After increasing the constant to deal with the finitely many cases where
\eqref{eq:log-bound} is unavailable, there are constants $a_q>0$ and
$c_q>0$ such that
\[
  P\geq c_q(n+2)^{-a_q}.
  \tag{4.9}\label{eq:P-lower}
\]
The same estimates, without logarithms, give a constant $b_q>0$ such that
\[
  \mu\leq b_qt.
  \tag{4.10}\label{eq:mu-upper}
\]

\medskip
\noindent
\textbf{Truncated inclusion--exclusion for the low moduli.}
Choose $K_q>a_q+2$, $R_q$ sufficiently large that
\[
  \frac{e b_q}{R_q}\leq\frac{1}{8}
  \qquad\text{and}\qquad
  \frac{R_q\log 4}{\log q}>a_q+2.
  \tag{4.11}\label{eq:R-choice}
\]
Let $r$ be the largest odd integer not exceeding $R_qt$.  Choose
$H_q\geq (K_q+1)R_q$, so that $H_q\cdot t^2\geq R_q t\cdot (K_q+1)t\geq R_qt\cdot \lceil K_q t\rceil$, therefore \eqref{eq:sieve-range} implies
\[
  rZ\leq m.
  \tag{4.12}\label{eq:level}
\]

For $j\leq r$, every set of $j$ low moduli has total degree at most
$jZ\leq m$.  Hence \eqref{eq:limited-independence} applies to every such
intersection.

Let $e_j(\lambda)$ be the $j$-th elementary symmetric polynomial in
the numbers $\lambda_p$ for the low moduli.  Bonferroni's inequality,
with $r$ odd, gives
\[
  \PP(\text{none of the low }E_p\text{ occurs})=1-\PP(\bigcup_{d_p\leq Z} E_p)
  \geq
  1-(\sum_{j=1}^r(-1)^{j-1} e_j(\lambda))=
  \sum_{j=0}^{r}(-1)^je_j(\lambda).
  \tag{4.13}\label{eq:bonferroni}
\]
On the other hand,
\[
  P=\prod_{d_p\leq Z}(1-\lambda_p)
   =\sum_{j\geq 0}(-1)^je_j(\lambda).
\]
Since $e_j(\lambda)\leq\mu^j/j!$,
\[
  \sum_{j>r}e_j(\lambda)
  \leq
  \sum_{j>r}\frac{\mu^j}{j!}.
  \tag{4.14}\label{eq:poisson-tail}
\]

For sufficiently large $n$, equations
\eqref{eq:mu-upper}--\eqref{eq:R-choice} imply
\[
  \frac{e\mu}{r+1}\leq\frac{eb_qt}{R_qt-1}\leq\frac{1}{4}
  \qquad\text{and}\qquad
  \frac{\mu}{r+2}\leq\frac12.
\]
Using $j!\geq(j/e)^j$ for the first term and 
\[
\frac{\mu^{j+1}}{(j+1)!}/\frac{\mu^j}{j!}=\frac{\mu}{j+1}\leq\frac{\mu}{r+2}\leq \frac{1}{2}
\]
gives
\[
  \sum_{j>r}\frac{\mu^j}{j!}
  \leq
  2\left(\frac{e\mu}{r+1}\right)^{r+1}
  \leq 2\cdot 4^{-(r+1)}.
  \tag{4.15}\label{eq:tail-final}
\]
By \eqref{eq:R-choice} and \eqref{eq:P-lower}, after increasing the lower
threshold for $n$ if necessary,
\[
  2\cdot 4^{-(r+1)}\leq\frac14P.
  \tag{4.16}\label{eq:tail-vs-P}
\]
Combining \eqref{eq:bonferroni}--\eqref{eq:tail-vs-P} yields
\[
  \PP(\text{none of the low }E_p\text{ occurs})
  \geq \frac34P.
  \tag{4.17}\label{eq:avoid-low}
\]

\medskip
\noindent
\textbf{The high moduli.}
If $Z<d_p$, we have
\[
  \PP(E_p)\leq k_pq^{\max\{-d_p,-m\}}\leq k_p q^{-Z}.
\]
The union bound therefore gives
\[
  \PP(\text{some high }E_p\text{ occurs})
  \leq nq^{-Z}.
  \tag{4.18}\label{eq:high-bound}
\]
Recall $K_q>a_q+2$.  Since $Z\geq K_qt$, equation
\eqref{eq:P-lower} implies, for sufficiently large
$n$,
\[
  nq^{-Z}\leq\frac14P.
  \tag{4.19}\label{eq:high-vs-P}
\]

Finally, \eqref{eq:avoid-low} and \eqref{eq:high-vs-P} give
\[
  \PP(\text{no }E_p\text{ occurs})
  \geq
  \frac34P-\frac14P
  =
  \frac12P>0.
\]
Thus some $f\in\V_m$ avoids all the sets $B_p$.  All estimates used above
hold once $n$ exceeds a constant depending only on $q$; this defines
$N_q$.
\end{proof}

\section{Proof of the main theorem}\label{pfsec}

\begin{proof}[Proof of Theorem~\ref{thm:main}]
Lemma~\ref{lem:sieve} gives constants $H_q$ and $N_q$ such that its
conclusion holds for $n\geq N_q$ whenever
\[
  m\geq H_q\left\lceil\log_q(n+2)\right\rceil^2.
\]
Increase $N_q$, if necessary, so that
\[
  \frac ns
  \geq
  H_q\left\lceil\log_q(n+2)\right\rceil^2
  \qquad(n>N_q).
  \tag{5.1}\label{eq:linear-beats-log}
\]
Set
\[
  C_q:=N_q.
\]

Suppose, for a contradiction, that there is a family of $n$ classes which
covers
\[
  \V_{L_{C_q}(n)},
  \qquad
  L_{C_q}(n)=\left\lceil\frac ns\right\rceil+C_q,
\]
but does not cover $\R$.  Choose $n$ minimal.

If $n\leq N_q$, then
\[
  L_{C_q}(n)\geq C_q=N_q\geq n.
\]
The family therefore covers $\V_n$, contradicting
Theorem~\ref{thm:known}.

Hence $n>N_q$.  Lemma~\ref{lem:reduction} produces a counterexample with
irreducible moduli and forbidden residue sets $B_p$ satisfying
\[
  |B_p|\leq s\deg p,
  \qquad
  \sum_p|B_p|=n.
\]
By \eqref{eq:linear-beats-log},
\[
  L_{C_q}(n)\geq\frac ns
  \geq
  H_q\left\lceil\log_q(n+2)\right\rceil^2.
\]
Lemma~\ref{lem:sieve}, applied with $m=L_{C_q}(n)$, gives an element of
$\V_{L_{C_q}(n)}$ which avoids every selected residue.  This contradicts
the assumed coverage.  Therefore
\[
  D_q(n)\leq
  \left\lceil\frac ns\right\rceil+C_q-1.
  \tag{5.2}\label{eq:upper-bound}
\]

It remains to prove the lower bound.  Put
\[
  M:=\left\lfloor\frac{n-1}{s}\right\rfloor.
\]
Consider the $sM+1$ classes
\[
  cx^j\pmod{x^{j+1}},
  \qquad
  0\leq j<M,\quad c\in\F_q^\times,
  \tag{5.3}\label{eq:nested-classes}
\]
together with
\[
  0\pmod{x^{M+1}}.
  \tag{5.4}\label{eq:zero-class}
\]
Every nonzero $f\in\V_M$ has a least index $j<M$ at which its
coefficient is nonzero, and then $f$ lies in the corresponding class in
\eqref{eq:nested-classes}.  The zero polynomial lies in
\eqref{eq:zero-class}.  Hence all of $\V_M$ is covered.

The polynomial $x^M$ lies in none of these classes.  If $sM+1<n$, add
$n-sM-1$ distinct classes
\[
  0\pmod{p_i},
\]
where the $p_i$ are distinct irreducibles of degree greater than $M$.
These extra classes also avoid $x^M$.  Thus there is a non-covering family
of exactly $n$ classes which covers $\V_M$, and so
\[
  D_q(n)\geq M
  =
  \left\lfloor\frac{n-1}{s}\right\rfloor.
  \tag{5.5}\label{eq:lower-bound}
\]

Combining \eqref{eq:upper-bound} and \eqref{eq:lower-bound} proves the
theorem.
\end{proof}

\section*{Declaration of generative AI use}

During the preparation of this work, the author used OpenAI’s ChatGPT 5.6 SOL to search for counterexamples to the earlier conjecture and to explore possible proof strategies. The system produced the example recorded as Example 3 and suggested the truncated CRT/Brun-sieve approach used in Lemma 7. The author independently checked and corrected the generated material, supplied the missing arguments, and takes full responsibility for the content of the manuscript. The tool was also used for language editing.

\bibliography{citations}
\bibliographystyle{plain}
\end{document}